\documentclass[12pt, twoside, leqno]{article}

\usepackage{amsmath,amsthm,amsrefs}
\usepackage{amssymb}
\usepackage{stmaryrd}
\usepackage{amscd}  
\usepackage{mathrsfs}
\usepackage{nccmath} 
\usepackage{lipsum}

\usepackage{enumerate}

\newtheorem{teo}{Theorem}[section]
\newtheorem{cor}[teo]{Corollary}
\newtheorem{lemma}[teo]{Lemma}
\newtheorem{prop}[teo]{Proposition}

\theoremstyle{definition}
\newtheorem{defn}[teo]{Definition}
\newtheorem{remark}[teo]{Remark}
\newtheorem{example}[teo]{Example}

\numberwithin{equation}{section}

\begin{document}


\baselineskip=17pt


\title{CM points on Shimura curves and $p$-adic binary quadratic forms}

\author{Piermarco Milione\\
Aalto University, School of Science\\ 
Dep. of Mathematics and Systems Analysis\\
Espoo, Finland\\
E-mail: piermarco.milione@aalto.fi
}

\date{}

\maketitle


\renewcommand{\thefootnote}{}



\renewcommand{\thefootnote}{\arabic{footnote}}
\setcounter{footnote}{0}


\begin{abstract}
	We prove that the set of CM points on the Shimura curve associa\-ted to an Eichler order inside an indefinite quaternion $\mathbb{Q}$-algebra, is in bijection with the set  of certain classes of $p$-adic binary quadratic forms, where $p$ is a prime dividing the discriminant of the quaternion algebra.
	The classes of $p$-adic binary quadratic forms are obtained by the action of a discrete and cocompact subgroup of $\mathrm{PGL}_{2}(\mathbb{Q}_{p})$ arising from the $p$-adic uniformization of the Shimura curve. 
	 We finally compute families of $p$-adic binary quadratic forms associated to an infinite family of Shimura curves studied in \cite{AmorosMilione2017}. 
	  	This extends results of Alsina-Bayer \cite{AlsinaBayer2004} to the $p$-adic context.
\end{abstract}

\section*{Introduction} 
	The theory of complex multiplication for the field of rational numbers $\mathbb{Q}$, provides a concrete way to generate ring class fields of imaginary quadratic fields $K$ over $\mathbb{Q}$. These extensions are obtained by adjoining to $K$ ``special values'', known as \emph{singular moduli}, of automorphic functions uniformizing Shimura curves (cf. \cite[Main Theorem I]{Shimura1967}). 
		Therefore, it is of some interest to understand the nature of these special algebraic points, with respect to the uniformization of the curve. 
			To be more precise: if $(X,J_{\Gamma})$ is the canonical model over $\mathbb{Q}$ of the Shimura curve $X$, together with its complex uniformization
\[
	J_{\Gamma}:\Gamma\backslash \mathcal{H}\simeq X(\mathbb{C}),
\]
and $\tau\in\Gamma\backslash\mathcal{H}$ is a parameter such that $J_{\Gamma}(\tau)$ is the special value in question, i.e. $K(J_{\Gamma}(\tau))$ is a certain ring class field of the imaginary quadratic field $K$, then we wish to understand how to obtain the parameter $\tau$, which is what we call a complex multiplication (CM) parameter by the field $K$.

	Complex multiplication parameters are fundamental in the explicit computation of the canonical model of a Shimura curve, since they characterize the uniformizing functions $J_{\Gamma}$ (cf. \cite{BayerTravesa2007}, \cite{BayerTravesa2008}, \cite{Nualart_tesis}, \cite{Voight2006},\cite{VoightWillis2011}). 
	For this reason it is necessary, in making explicit a complex uniformization of the Shimura curve $X$, to place these parameters inside the fundamental domain associated to the uniformization (cf. \cite{AlsinaBayer2004},\cite{BayerRemon2014},\cite{Voight2006}). 
	Finally the computation of CM parameters can be used in many cases to compute algebraic points on modular elliptic curves (see \cite{DarmonBook} for an overview on this subject).

	On the one hand, CM parameters have a geometric interpretation, coming from the property of the curve $X$ to be a coarse moduli space of families of abelian varieties with some additional structure (see \cite{BayerGuardia2005} for more details and examples). 
	This modular interpretation is actually the key to understanding the theory of complex multiplication as well as the proofs of its statements.
	On the other hand, CM parameters also arise as zeros of certain classes of integral binary quadratic forms. 
	These quadratic forms have rational integral coefficients when we are looking for complex multiplication parameters arising from Shimura curves of discriminant 1 (i.e. classical modular curves): this is the celebrated theory of binary quadratic forms started in Gau\ss' \emph{Disquisitiones Arithmetiquae}, (cf. \cite{BullBook}, \cite{CoxBook}).
		 And they have quadratic integral coefficients when classical modular curves are replaced by general Shimura curves of arbitrary discriminant (cf. \cite{AlsinaBayer2004}). 

	All these theories of binary quadratic forms are associated to the complex  (and therefore archimedean) uniformization of Shimura curves, so it seems quite natural to investigate the analogous situation for the non-archimedean  uniformization, when this is available. 
	Thanks to the Cerednik-Drinfeld theory (cf. \cite{BoutotCarayol1991}) we have a description of the $p$-adic points of a Shimura curve $X$ over $\mathbb{Q}$ associated to an indefinite quaternion algebra of discriminant $Dp$, where $p$ is a distinguished prime number. 
	In particular, we have a bijection on the set of $\mathbb{Q}_{p^{2}}$-points of the Shimura curve $X$, where $\mathbb{Q}_{p^{2}}$ denotes the quadratic unramified extension of $\mathbb{Q}_{p}$. The bijection is between the following sets:
\[
	\Gamma_{p,+}\backslash\mathcal{H}_{p}(\mathbb{Q}_{p^{2}})\simeq 	X(\mathbb{Q}_{p^{2}}),
\]
	where $\mathcal{H}_{p}(\mathbb{Q}_{p^{2}})=\mathbb{P}^{1}(\mathbb{Q}_{p^{2}})\smallsetminus\mathbb{P}^{1}(\mathbb{Q}_{p})$ is the set of $\mathbb{Q}_{p^{2}}$-valued points of the $p$-adic upper half-plane $\mathcal{H}_{p}$, and $\Gamma_{p,+}$ is a discrete and cocompact subgroup of $\mathrm{PGL}_{2}(\mathbb{Q}_{p})$ which is explicitly constructed from the quaternion algebra defining the Shimura curve $X$. 

	In this paper we associate to the $p$-adic uniformization of the Shimura curve $X$ certain parameters in $\mathcal{H}_{p}(\mathbb{Q}_{p^{2}})$, analogous to the CM parameters in $\mathcal{H}$. Similar $p$-adic parameters have been already considered in \cite{Greenberg2009}, and the mod $p$ reduction of CM points on Shimura curves has been studied in \cite{Molina2012b}.

	We follow, in the study of these parameters, the $p$-adic analog of the line adopted in \cite{AlsinaBayer2004}.
	We are able to recover these parameters as zeros of certain binary quadratic forms with $p$-adic coefficients. 
	Finally, in Theorem \ref{numb_p-adic_cm}, we relate the number of $\Gamma_{p,+}$-classes in $\mathcal{H}_{p}(\mathbb{Q}_{p^{2}})$ of $p$-imaginary multiplication parameters with the number of $\Gamma$-classes in $\mathcal{H}$ of complex multiplication parameters, proving that these two cardinalities coincide.   

\subsection*{Acknowledgements} I wish to thank again Prof. Pilar Bayer for inspiring this work from the very beginning, and for sharing her knowledge on Shimura curves.

\subsection*{Notations}
	When $\ell$ is a fixed finite prime of $\mathbb{Q}$, we will denote by $\mathbb{Q}_{\ell^{2}}$ the quadratic unramified extension of $\mathbb{Q}_{\ell}$, which is unique inside a fixed algebraic closure $\overline{\mathbb{Q}}_{\ell}$ of $\mathbb{Q}_{\ell}$. 
	When $\ell=\infty$ is the archimedean prime of $\mathbb{Q}$, we use then the notations $\mathbb{Q}_{\infty}=\mathbb{R}$ and $\mathbb{Q}_{\infty^{2}}:=\mathbb{C}$, according to the convention for which $\mathbb{C}$ is the quadratic unramified extension of $\mathbb{R}$ (cf. \cite[Ch. III]{NeukirchBook}). 
	
	For every prime $\ell$ of $\mathbb{Q}$ we denote by $\mathbb{Z}[1/\ell]$  the ring of integers outside $\ell$, i.e.,
\[
\mathbb{Z}[1/\ell]:=\bigcap_{q\neq\ell}(\mathbb{Z}_{q}\cap\mathbb{Q}).
\]
	In particular when $\ell=\infty$, then $\mathbb{Z}[1/\infty]=\mathbb{Z}$.

\section{Transformations in $\mathrm{PGL}_{2}(\mathbb{Q}_{p})$}

We start by recalling how transformations in $\mathrm{PGL}_{2}(\mathbb{Q}_{p})$ are classified. For more details we refer the reader to \cite{Gerritzen_vanderPut1980}.

\begin{defn}
	Let $\gamma\in\mathrm{PGL}_{2}(\mathbb{Q}_{p})$ be a transformation represented by a matrix having eigenvalues $\mu_{1},\mu_{2}\in\overline{\mathbb{Q}}_{p}$.
	\begin{enumerate}[$(i)$]
		\item
		$\gamma$ is called hyperbolic if $|\mu_{1}|\neq|\mu_{2}|$.
		\item
		$\gamma$ is called elliptic if $\mu_{1}\neq\mu_{2}\; and\; |\mu_{1}|=|\mu_{2}|$.
		\item
		$\gamma$ is called parabolic if $\mu_{1}=\mu_{2}$.
	\end{enumerate}
\end{defn}

	In a similar way as in the case of $\mathrm{PSL}_{2}(\mathbb{R})$ (cf. \cite[Sec. 2.2]{AlsinaBayer2004}), we can associate to every matrix 
$\gamma=\left(\begin{array}{cc} a &b\\ 
c &d \end{array}\right)
\in\mathrm{M}_{2}(\mathbb{Q}_{p})$, the following $p$-adic binary quadratic form:
\[
f_{\gamma}(X,Y):=cX^{2}+(d-a)XY-bY^{2}\in\mathbb{Q}_{p}[X,Y].
\]
We call zeros of the quadratic form $f_{\gamma}$ the zeros of the quadratic polynomial $f_{\gamma}(X,1)\in\mathbb{Q}_{p}[X]$.

\begin{remark}
	Observe that the form $f_{\gamma}$ has the following properties:
\begin{enumerate}[($i$)]
\item
The discriminant of the polynomial $f_{\gamma}(X,1)$ is 
\[
\mathrm{disc}(f_{\gamma}(X,1))=(d-a)^{2}+4bc=\mathrm{Tr}^{2}(\gamma)-4\mathrm{Det}(\gamma),
\]
which turns out to be equal to the discriminant of the characteristic polynomial associated to $\gamma$.
\item
The zeros of the form $f_{\gamma}$ are the fixed points in $\mathcal{H}_{p}$ of the transformation represented by the matrix $\gamma$, i.e.
\[
f_{\gamma}(\tau,1)=0\iff \gamma\cdot\tau=\tau.
\]
\end{enumerate}
\end{remark}

\begin{prop}\label{caracterizacion_transf_formas}
The following characterizations hold:
\begin{enumerate}[$(i)$]
\item
$\gamma$ is hyperbolic $\iff\mathrm{disc}(f_{\gamma}(X,1))\in\mathbb{Q}_{p}^{*2}\iff\gamma$ has two different fixed points in $\mathbb{P}^{1}(\mathbb{Q}_{p})$. 
\item
$\gamma$ is elliptic $\iff\mathrm{disc}(f_{\gamma}(X,1))\notin\mathbb{Q}_{p}^{*2}\iff\gamma$ has two different fixed points in $\mathbb{Q}_{p^{2}}\smallsetminus \mathbb{Q}_p$. 
\item\
$\gamma$ is parabolic $\iff\mathrm{disc}(f_{\gamma}(X,1))=0\iff\gamma$ has a unique fixed point in $\mathbb{P}^{1}(\mathbb{Q}_{p})$.
\end{enumerate}
\end{prop}

\begin{remark}
	Recall that $\mathbb{Q}_{p^{2}}\smallsetminus \mathbb{Q}_p$ is the set of the $\mathbb{Q}_{p^{2}}$-points of the $p$-adic upper-half plane $\mathcal{H}_{p}$ (cf. \cite{AmorosMilione2017} for a brief overview).
	We wish to underline the difference between Proposition \ref{caracterizacion_transf_formas} and  its archimedean analog. In fact, the connectedness of the $p$-adic upper half-plane $\mathcal{H}_{p}$ reflects into the fact that elliptic transformations have two fixed points instead of one.
\end{remark}

\section{Embeddings of $p$-imaginary quadratic fields into definite quaternion algebras}

\begin{defn}
	Let $\ell$ be a (finite or archimedean) prime of $\mathbb{Q}$. We say that a quadratic field $K|\mathbb{Q}$ is {imaginary at $\ell$} (also {$\ell$-imaginary}) if $K\otimes_{\mathbb{Q}}\mathbb{Q}_{\ell}\simeq\mathbb{Q}_{\ell^{2}}$. 
\end{defn}

\begin{remark}
	In particular, a quadratic field $K$ is imaginary at the prime $\infty$ if and only if it is imaginary, in the usual sense. 
	Otherwise, if $\ell$ is a finite prime then a 	quadratic field $K$ is imaginary at the prime $\ell$ if and only if $K=\mathbb{Q}(\sqrt d)$, for a square-free integer $d$ coprime with $\ell$ and such that $\left(\frac{d}{\ell}\right)=-1$, that is if and only if the prime $\ell$ is inert in $K$.
\end{remark}

	Let $H$ be a definite quaternion $\mathbb{Q}$-algebra of discriminant $D$ and let $p\in\mathbb{Z}$ be a prime not dividing $D$.
Let us fix an isomorphism of $\mathbb{Q}_{p}$-algebras $\Phi_{p}:H_{p}:=H\otimes_{\mathbb{Q}}\mathbb{Q}_{p}\simeq\mathrm{M}_{2}(\mathbb{Q}_{p})$, and let us also denote by $\Phi_{p}$ the corresponding matricial immersion $H\hookrightarrow\mathrm{M}_{2}(\mathbb{Q}_{p})$.

	If $K$ is a $p$-imaginary quadratic field such that there exists an immersion $\varphi:K\hookrightarrow H$ of $\mathbb{Q}$-algebras, then we can consider the corresponding matricial immersion $\Phi_{p}\circ\varphi:K\hookrightarrow\mathrm{M}_{2}(\mathbb{Q}_{p})$ and the following set of transformations:
\[
\{\Phi_{p}(\varphi(a))\in\mathrm{GL}_{2}(\mathbb{Q}_{p})\mid\,a\in K^{*}\}.
\]

	We are now going to prove, in Proposition \ref{shimura_p-adico1}, the $p$-adic version of \cite[Proposition 2.6 ]{Shimura1967}. We first need the following easy lemma.

\begin{lemma}\label{lemma}
	Let $\gamma,\gamma'\in\mathrm{GL}_{2}(\mathbb{Q}_{p})$. Then $\gamma,\gamma'$ have the same fixed points if and only if there exist $\lambda,\mu\in\mathbb{Q}_{p}$, $\lambda\neq 0$, such that $\gamma'=\lambda\mathrm{I}_{2}+\mu\gamma$. 
\end{lemma}

\begin{proof}
	The proof goes as in \cite[6.1]{AlsinaBayer2004}, after replacing $\mathbb{R}$ by $\mathbb{Q}_{p}$.
\end{proof}

\begin{prop}\label{shimura_p-adico1}
	Let $K$ be a quadratic field imaginary at the prime $p$ and let $\varphi:K\hookrightarrow H$ be an immersion of $\mathbb{Q}$-algebras. Then for every $a\in K^{*}$ all the transformations $\Phi_{p}(\varphi(a))\in\mathrm{GL}_{2}(\mathbb{Q}_{p})$ are elliptic and they have two fixed points in $\mathcal{H}_{p}(\mathbb{Q}_{p^{2}})=\mathbb{Q}_{p^{2}}\smallsetminus\mathbb{Q}_{p}$ not depending on $a$.
\end{prop}

\begin{proof}
	Let us write $a=x+y\sqrt d\in K^{*}$. Then $\Phi_{p}(\varphi(a))=x\mathrm{I}_{2}+y\gamma\in\mathrm{GL}_{2}(\mathbb{Q}_{p})$, where 
$\gamma:=\Phi_{p}(\varphi({\sqrt d}))$ is a transformation such that $\gamma^{2}=d\mathrm{I}_{2}$. Hence, by Lemma \ref{lemma} it is clear that all the transformations $\Phi_{p}(\varphi(a))$ have the same fixed points as $\gamma$. Now it is easy to see that the discriminant of the characteristic polynomial of $\gamma=\Phi_{p}(\varphi({\sqrt d}))$ is
\[
\mathrm{tr}^{2}(\gamma)-4\mathrm{det}(\gamma)=4d\notin\mathbb{Q}_{p}^{*2}
\]
since $\left(\frac{d}{p}\right)=-1$. Finally, the transformation $\gamma$ is elliptic and its fixed points $z_{1},z_{2}$ are in $\mathbb{Q}_{p^{2}}\smallsetminus\mathbb{Q}_{p}=\mathcal{H}_{p}(\mathbb{Q}_{p^{2}})$. Actually, $z_{1},z_{2}\in\mathbb{Q}_{p}(\sqrt d)\simeq\mathbb{Q}_{p^{2}}$.
\end{proof}

The reciprocal is also true and it is the following:

\begin{prop}
	If $\alpha\in H,\alpha\notin\mathbb{Q}$, is such that the associated transformation $\Phi_{p}(\alpha)\in\mathrm{GL}_{2}(\mathbb{Q}_{p})$ has two fixed points $z_{1},z_{2}\in\mathcal{H}_{p}(\mathbb{Q}_{p^{2}})$, then $\mathbb{Q}(\alpha)$ is a quadratic extension of $\mathbb{Q}$, imaginary at the prime $p$, admitting an embedding in $H$. 
\end{prop}

\begin{proof}
	For every quaternion $\alpha\in H$ not in $\mathbb{Q}$, the extension $\mathbb{Q}(\alpha)$ is a quadratic subfield of the algebra $H$, since $\alpha$ satisfies the quadratic equation $X^{2}-\mathrm{Tr}_{H/\mathbb{Q}}(\alpha)X+\mathrm{Nm}_{H/\mathbb{Q}}(\alpha)=0$. Now if we set $\gamma:=\Phi_{p}(\alpha)\in\mathrm{GL}_{2}(\mathbb{Q}_{p})$ then $\mathbb{Q}(\alpha)=\mathbb{Q}(\sqrt{d})$ where $d$ is the discriminant of the characteristic polynomial $P_{\gamma}(X)$ of $\gamma$. Since the transformation $\gamma$ is elliptic, by Proposition \ref{caracterizacion_transf_formas} we have that $d\notin\mathbb{Q}_{p}^{*2}$ and so $\mathbb{Q}_{p}(\sqrt d)\simeq\mathbb{Q}_{p^{2}}$. Hence $\mathbb{Q}(\alpha)$ is a $p$-imaginary quadratic subfield of $H$.
\end{proof}

\begin{remark}\label{expression_fixed_points}
	If $\gamma:=\Phi_{p}(\varphi({\sqrt d}))=\left(\begin{matrix} a & b \\ c & -a\end{matrix}\right)$, the zeros of its associated binary quadratic form have the following expressions:
\[
\begin{array}{lcr}
z_{1}=\dfrac{a+\sqrt{d}}{c},&\,&z_{2}=\dfrac{a-\sqrt{d}}{c}.
\end{array}
\]
	Therefore we see that $z_{1}$ and $z_{2}$ are Galois conjugated in the quadratic extension $\mathbb{Q}_{p}(\sqrt{d})$, so we can write $z_{1}=\overline{z}_{2}$ and we have that the action of $\gamma$ on these points is
\[
\begin{array}{lcr}
\gamma\left(\begin{matrix} z_{1}\\ 1\end{matrix}\right)=\sqrt{d}\,\left(\begin{matrix} z_{1}\\ 1\end{matrix}\right),&\,&\gamma\left(\begin{matrix}\overline{z}_{1}\\ 1 \end{matrix}\right)=-\sqrt{d}\,\left(\begin{matrix} z_{2}\\ 1\end{matrix}\right).
\end{array}
\]
\end{remark}

	Following the same spirit as in \cite{AlsinaBayer2004} we want to relate the set of embeddings of $p$-imaginary quadratic fields into the quaternion algebra $H$ with the set of representations of an integer by certain quadratic forms associated to the algebra.

\begin{defn}
	An integer $t$ is said to be represented by an $n$-ary quadratic form 
	$$f\in\mathbb{Z}[1/\ell][X_{1},\dots,X_{n}]$$
	 if there exist $a_{1},\dots,a_{n}\in\mathbb{Z}[1/\ell]$ such that 
	 $f(a_{1},\dots,a_{n})=t.$
	In this case we write $t\leftarrow f,$ and the $n$-tuple $(a_{1},\dots,a_{n})$ is called a {representation}. The representation is said to be {primitive} if the ideal generated by $a_{1},\dots,a_{n}$ is $(a_{1},\dots,a_{n})=\mathbb{Z}[1/\ell]$. 

	We will denote by 
	$$\mathrm{Rep}(f,t;\mathbb{Z}[1/\ell]),\;\mathrm{and}\; \mathrm{Rep}^{*}(f,t;\mathbb{Z}[1/\ell])$$ 
	the set of representations of the integer $t$ by the form $f$, and its subset of primitive representations, respectively. 
\end{defn}

\begin{defn}
	Let $Q$ be a quaternion $\mathbb{Q}$-algebra, $K$ a quadratic field, and $\ell$ a prime integer. 
	If $\mathcal{O}_{K,m}[1/\ell]\subseteq K$ is a quadratic order over $\mathbb{Z}[1/\ell]$ of conductor $m$ and $\mathcal{O}[1/\ell]$ is an Eichler order over $\mathbb{Z}[1/\ell]$ in $Q$, then we will denote by 
	\[
	\mathcal{E}(K,Q)\quad \mathrm{and}\quad \mathcal{E}(\mathcal{O}_{K,m}[1/\ell],\mathcal{O}[1/\ell]),
	\] 
	respectively, the set of embeddings $K\hookrightarrow Q$ of $\mathbb{Q}$-algebras and the set of embeddings $\mathcal{O}_{K,m}[1/\ell]\hookrightarrow\mathcal{O}[1/\ell]$ of $\mathbb{Z}[1/\ell]$-modules.
\end{defn}

\begin{defn}[cf.\cite{AlsinaBayer2004}, Definition 4.7 ]
	Let $\ell$ be a fixed prime integer. Let $Q$ be a quaternion $\mathbb{Q}$-algebra, $\mathcal{O}[1/\ell]$ be an Eichler order over $\mathbb{Z}[1/\ell]$ and $\mathcal{O}_{K,m}[1/\ell]$ be an order over $\mathbb{Z}[1/\ell]$ inside a quadratic field $K$ and of conductor $m$. 

	An embedding $\varphi:\mathcal{O}_{K,m}[1/\ell]\hookrightarrow\mathcal{O}[1/\ell]$ is said to be {optimal} if it satisfies the equality $\varphi(\mathcal{O}_{K,m}[1/\ell])=\mathcal{O}[1/\ell]\cap\varphi(K)$.

	We will denote by 
\[
\mathcal{E}^{*}(\mathcal{O}_{K,m}[1/\ell],\mathcal{O}[1/\ell])
\] 
the subset of $\mathcal{E}(\mathcal{O}_{K,m}[1/\ell],\mathcal{O}[1/\ell])$ of optimal embeddings.
\end{defn}

\begin{prop}\label{embeddings_rep}
	Let $H$ be a definite quaternion algebra over $\mathbb{Q}$.

	If $K=\mathbb{Q}(\sqrt d)$ is a quadratic field, with $d$ a square-free integer, then there exists a bijection of sets
	\[
	\mathcal{E}(K,H)\simeq\mathrm{Rep}(\mathrm{N}_{H,3},-d;\mathbb{Q}).
	\]
	In particular, every quadratic field $K$ admitting an embedding $K\hookrightarrow H$ is imaginary at $\infty$, i.e. $d<0$.
\end{prop}

\begin{proof}
	The proof is the same as in \cite[4.2]{AlsinaBayer2004}. The idea is that if $\varphi :K\hookrightarrow H$ is an embedding, then $\varphi(\sqrt{d})$ is a pure quaternion of norm $\mathrm{Nm}_{H/\mathbb{Q}}(\varphi(\sqrt{d}))=\mathrm{Nm}_{K/\mathbb{Q}}(\sqrt{d})=-d$ and so its coordinates in a $\mathbb{Q}$-basis of $H$ gives the desired representation. The viceversa is then clear. 

	Finally, since the ternary normic form $\mathrm{N}_{H,3}$ is positive definite, condition $-d\leftarrow\mathrm{N}_{H,3}$ implies that $d$ is negative. 
\end{proof}

\begin{prop}
	Let $H$ be a definite quaternion algebra over $\mathbb{Q}$ of discriminant $D$ and let $p$ be a prime integer not dividing $D$. Let $B$ be an indefinite quaternion algebra over $\mathbb{Q}$ of discriminant $Dp$ and let $K$ be a quadratic field imaginary at both the primes $p$ and $\infty$. Then there is a bijection of sets 
	\[
	\mathcal{E}(K,B)\simeq\mathcal{E}(K,H).
	\] 
\end{prop}

\begin{proof}
	First recall that for the indefinite quaternion algebra $B$ the statement analogous to Proposition \ref{embeddings_rep} holds (cf. \cite[4.2]{AlsinaBayer2004}) so each set of embeddings of the statement can be described by representations. 

	Let us start by noting that the quaternion algebras $B$ and $H$ are such that $B_{\ell}\simeq H_{\ell}$ for every $\ell\notin\{\infty,p\}$, so in order to apply the Principle of Hasse-Minkowski for quadratic forms or, equivalently, for embeddings (cf. \cite[3.2]{Vigneras1980}) we only need to look at the behavior of the immersion/representation at one of the primes $p,\infty$ (the other one comes for free via the product formula for the Hilbert symbol). 

	From left to right, we have that the ternary normic form $\mathrm{N}_{3,H}$ always represents all positive integers over $\mathbb{R}$. 
	From right to left, we have that the local quaternion $\mathbb{Q}_{p}$-algebra $B_{p}$ is a field containing two copies of $\mathbb{Q}_{p^{2}}\simeq K_{(p)}$, since $B$ is ramified at $p$.
\end{proof}

So we have seen that the set of quadratic fields 
\[
\{K\mid\,[K:\mathbb{Q}]=2,\,K\hookrightarrow H,\;K\otimes\mathbb{Q}_{p}\simeq\mathbb{Q}_{p^{2}}\}
\]
is included in the set
\[
\{K\mid\,[K:\mathbb{Q}]=2,\,K\hookrightarrow B,\;K\otimes\mathbb{R}\simeq\mathbb{C}\}
\]
and the inclusion is strict since there are negative integers $d$ such that $\left(\frac{d}{p}\right)\neq -1$.

	From now on, $H$ and $B$ will denote, respectively, a definite and an indefinite quaternion algebra over $\mathbb{Q}$ of discriminant $D$ and $Dp$, for a fixed prime integer prime integer $p\nmid D$. And $\mathcal{O}_{B}\subseteq B$ and $\mathcal{O}\subseteq H$ will denote Eichler orders over $\mathbb{Z}$ of the same level $N$.

	Let $K=\mathbb{Q}(\sqrt{d})$ be a quadratic field, where $d$ is a square-free integer, and let $\mathcal{O}_{K,m}\subseteq K$ be an order over $\mathbb{Z}$ of conductor $m$. As usual, we denote by $\mathcal{O}[1/p]:=\mathcal{O}\otimes_{\mathbb{Z}}\mathbb{Z}[1/p]$ and $\mathcal{O}_{K,m}[1/p]:=\mathcal{O}_{K,m}\otimes_{\mathbb{Z}}\mathbb{Z}[1/p]$ the corresponding $\mathbb{Z}[1/p]$-orders in $H$ and in $K$, respectively.

	We will consider the following subgroups of quaternion units, arising from the complex and the $p$-adic uniformizations of the Shimura curve $X(Dp,N)$ associated to the Eichler order $\mathcal{O}_{B}$: 
	\begin{enumerate}[(i)]
		\item $\mathcal{O}_{B,+}^{*}:=\{\alpha\in\mathcal{O}^{*}_{B}\mid\,\mathrm{Nm}(\alpha)>0\}=\{\alpha\in\mathcal{O}^{*}_{B}\mid\,\mathrm{Nm}(\alpha)=1\}$,
		\item $\mathcal{O}[1/p]^{*}_{+}:=\{\alpha\in\mathcal{O}[1/p]^{*}\mid v_{p}(\mathrm{Nm}(\alpha))\equiv 0\,(\mathrm{mod}\,2)\}$.
\end{enumerate}

	Thus we have right actions of these groups over the sets 
	$$\mbox{$\mathcal{E}^{*}(\mathcal{O}_{K,m},\mathcal{O}_{B})$ and $\mathcal{E}^{*}(\mathcal{O}_{K,m}[1/p],\mathcal{O}[1/p])$ respectively,}$$ 
	as defined in the following:
	\begin{enumerate}[(i)] 
		\item For any $\alpha\in \mathcal{O}_{B,+}^{*}$ and any $\varphi\in\mathcal{E}^{*}(\mathcal{O}_{K,m},\mathcal{O}_{B})$,
\[
	(\varphi\cdot\alpha)(x):=\alpha^{-1}\varphi(x)\alpha,\;x\in\mathcal{O}_{K,m}.
\]
		\item For any $\alpha\in \mathcal{O}[1/p]^{*}_{+}$ and any $\varphi\in\mathcal{E}^{*}(\mathcal{O}_{K,m}[1/p],\mathcal{O}[1/p])$,
\[
	(\varphi\cdot\alpha)(x):=\alpha^{-1}\varphi(x)\alpha,\;x\in\mathcal{O}_{K,m}[1/p].
\]
	\end{enumerate}

	Therefore we can define the following cardinalities of sets, following the same notations as in \cite[4.2]{AlsinaBayer2004}:

\begin{enumerate}[(i)]
	\item
$\nu(Dp,N,d,m; \mathcal{O}_{B,+}^{*}):=\#\mathcal{E}^{*}(\mathcal{O}_{K,m},\mathcal{O}_{B})/\mathcal{O}_{B,+}^{*}$,
	\item
$\nu(D,N,d,m; \mathcal{O}[1/p]^{*}_{+}):=\#\mathcal{E}^{*}(\mathcal{O}_{K,m}[1/p],\mathcal{O}[1/p])/\mathcal{O}[1/p]^{*}_{+}$.
\end{enumerate} 

\begin{remark}
	The cardinality $\nu(D,N,d,m; \mathcal{O}[1/p]^{*}_{+})$ does not depend on the type of the order $\mathcal{O}[1/p]$ in $H$ and the cardinality $\nu(Dp,N,d,m; \mathcal{O}_{B,+}^{*})$ does not depend on the type of the order $\mathcal{O}_{B}$ in $B$, since the number of types $t(D,N)$ and $t(Dp,N)$ are both equal to 1 (as a consequence of \cite[Corollaire 5.7]{Vigneras1980}).
\end{remark}

	In the same way we can define also the cardinalities 
		\begin{enumerate}[(i)]
		\item
		$\nu(Dp,N,d,m; \mathcal{O}_{B}^{*}):=\#\mathcal{E}^{*}(\mathcal{O}_{K,m},		\mathcal{O}_{B})/\mathcal{O}_{B}^{*}$,
		\item
		$\nu(D,N,d,m; \mathcal{O}[1/p]^{*}):=\#\mathcal{E}^{*}(\mathcal{O}_{K,m}[1/p],\mathcal{O}[1/p])/\mathcal{O}[1/p]^{*}$.
		\end{enumerate}

	In order to compute these cardinalities we have to define the following associated local factors:
		\begin{enumerate}[(i)] 
		\item
		For every prime integer $\ell$,
		\[
		\nu_{\ell}(Dp,N,d,m; \mathcal{O}_{B}^{*}):=\#\mathcal{E}^{*}(\mathcal{O}_{K,\ell},\mathcal{O}_{B})/\mathcal{O}_{B,\ell}^{*}.
		\]
		\item
		For every prime integer $\ell\neq p$,
		\[
		\nu_{\ell}(D,N,d,m; \mathcal{O}[1/p]^{*}):=\#\mathcal{E}^{*}(\mathcal{O}_{K,\ell},\mathcal{O}_{\ell})/\mathcal{O}_{\ell}^{*}.
		\]
		\end{enumerate}

	\begin{teo}\label{number_embeddings}
		Let $K=\mathbb{Q}(\sqrt{d})$ be a quadratic field, with $d$ a square-free integer, and let $\mathcal{O}_{K,m}\subseteq K$ be an order over $\mathbb{Z}$ of conductor $m$. 

	Then the cardinalities 
		$$\begin{array}{l}
			\nu(Dp,N,d,m; \mathcal{O}_{B}^{*}),\\
		\nu(D,N,d,m; \mathcal{O}[1/p]^{*}), \\
		\nu(Dp,N,d,m; \mathcal{O}_{B,+}^{*}), \\
		\nu(D,N,d,m; \mathcal{O}[1/p]^{*}_{+})
		\end{array}$$
		are finite. 
	Moreover if the quadratic field $K$ is imaginary at both $p$ and $\infty$, then the following relations are satisfied:
\begin{enumerate}[$(i)$]
\item
	$\nu(D,N,d,m; \mathcal{O}[1/p]^{*}_{+})=
	2\nu(D,N,d,m; \mathcal{O}[1/p]^{*})$,
\item
	$\nu(Dp,N,d,m; \mathcal{O}_{B,+}^{*})=
	2\nu(Dp,N,d,m; \mathcal{O}_{B}^{*})$,
\item
	$\nu(Dp,N,d,m; \mathcal{O}_{B}^{*})=
	2\nu(D,N,d,m; \mathcal{O}[1/p]^{*})$,
\item
	$\nu(Dp,N,d,m; \mathcal{O}_{B,+}^{*})=
	2\nu(D,N,d,m; \mathcal{O}[1/p]^{*}_{+})$.
\end{enumerate}
\end{teo}

\begin{proof}
	We start with the proof of point ($iii$). Following \cite[Theorem 3.1, 3.2 ]{Vigneras1980} we have that when $N$ is square-free,
\[
\nu_{\ell}(Dp,N,d,m; \mathcal{O}_{B}^{*})=
\left\{\begin{array}{ll}
1-\left(\dfrac{d}{\ell}\right),   & \mathrm{if}\,\ell\mid Dp \\
\, & \,\\
1+\left(\dfrac{d}{\ell}\right),  & \mathrm{if}\,\ell\mid N \\
\, & \, \\
1, & \mathrm{otherwise},
\end{array}\right.
\]
and for every $\ell\neq p$,
\[
\nu_{\ell}(D,N,d,m; \mathcal{O}[1/p]^{*})=
\left\{\begin{array}{ll}
1-\left(\dfrac{d}{\ell}\right), & \mathrm{if}\,\ell\mid D\\
\, & \,\\
1+\left(\dfrac{d}{\ell}\right),  & \mathrm{if}\,\ell\mid N\\
\, & \,\\
1, & \mathrm{otherwise}.
\end{array}\right.\]

	In particular note that 
\[
\nu_{\ell}(Dp,N,d,m; \mathcal{O}_{B}^{*})=
\nu_{\ell}(D,N,d,m; \mathcal{O}[1/p]^{*})
\] 
for every $\ell\neq p$ since $B_{\ell}\simeq H_{\ell}$ for every finite prime $\ell\neq p$.

	Now by \cite[Theorem 4.19]{AlsinaBayer2004} we already know that $\nu(Dp,N,d,m; \mathcal{O}_{B}^{*})$ is finite. This is actually \cite[Theorem 5.15]{Vigneras1980} since the Eichler order $\mathcal{O}_{B}$ over $\mathbb{Z}$ satisfies Eichler's condition (cf. \cite[p. 81]{Vigneras1980}). Moreover, we can again apply \cite[Theorem 5.15]{Vigneras1980} to the Eichler order $\mathcal{O}[1/p]$ over $\mathbb{Z}[1/p]$ and, by the computations of the local factors above, we find the following equalities:
\small
\begin{multline*}
\nu (D,N,d,m; \mathcal{O}[1/p]^{*})=h(\mathcal{O}_{K,m}[1/p])\prod_{\ell\neq p}\nu_{\ell}(D,N,d,m; \mathcal{O}[1/p]^{*})=\\
h(\mathcal{O}_{K,m})\prod_{\ell\neq p}\nu(Dp,N,d,m;\mathcal{O}_{B,\ell}^{*})= \nu(Dp,N,d,m; \mathcal{O}_{B}^{*})\nu_{p}(Dp,N,d,m; \mathcal{O}_{B}^{*})^{-1}=\\
=\frac{1}{2}\nu(Dp,N,d,m; \mathcal{O}_{B}^{*}).
\end{multline*}
\normalsize

	Note that since $K$ is $p$-imaginary, $\nu_{p}(Dp,d,m; \mathcal{O}_{B}^{*})=2$, and the ideal class numbers of the orders $\mathcal{O}_{K,m}[1/p]$ and $\mathcal{O}_{K,m}$ coincide. This completes the proof of the equality ($iii$).

	When $N$ is not square-free, the corresponding formulas for the local cardinalities are given in \cite[Theorem 4.19]{AlsinaBayer2004} (see\cite{Eichler1955_crelle} for a proof) and again these formulas depend only on the primes $\ell\vert N, \ell\neq p$ and on the integer $d$. Hence the proof of equality ($iii$) in this case proceeds in the same way.

	The first (resp. the second) equality of the statement follows from the simple observation that the group $\mathcal{O}[1/p]^{*}_{+}$ (resp. $\mathcal{O}_{B,+}^{*}$) has index $2$ inside $\mathcal{O}[1/p]^{*}$ (resp. $\mathcal{O}_{B}^{*}$). Finally ($iv$) is a consequence of ($i$), ($ii$) and ($iii$).
\end{proof}

\section{$p$-adic binary quadratic forms associated to Shimura curves} 
Let $H$ be a definite quaternion algebra over $\mathbb{Q}$ of discrimi\-nant $D$ and let $p$ be a prime integer such that $p\nmid D$. Let $\mathcal{O}\subseteq H$ be an Eichler order of level $N$ coprime with $p$ and denote by $\mathcal{O}[1/p]:=\mathcal{O}\otimes\mathbb{Z}[1/p]$ the corresponding $\mathbb{Z}[1/p]$-order. Let $K$ be a $p$-imaginary quadratic field of discriminant $D_{K}$ and let $\mathcal{O}_{K,m}[1/p]$ be a $\mathbb{Z}[1/p]$-order of conductor $m$. 

\begin{teo}\label{bij_emb_rep}
	If we denote by $\mathcal{O}'$ the order of $H$ equal to $\mathbb{Z}+2\mathcal{O}$, then there is a bijective map
\[
\rho :\mathcal{E}(\mathcal{O}_{K,m}[1/p],\mathcal{O}[1/p]) \rightarrow\mathrm{Rep}(\mathrm{N}_{\mathcal{O}',3},-m^{2}D_{K};\mathbb{Z}[1/p]),
\]
whose restriciton to the subset of optimal embeddings $\mathcal{E}^{*}(\mathcal{O}_{K,m}[1/p],\mathcal{O}[1/p])$ induces a bijection
\[
\rho^{*}:\mathcal{E}^{*}(\mathcal{O}_{K,m}[1/p],\mathcal{O}[1/p])\rightarrow\mathrm{Rep}^{*}(\mathrm{N}_{\mathcal{O}',3}, -m^{2}D_{K};\mathbb{Z}[1/p]).
\]

\end{teo}

\begin{proof}
	The proof is the same as the one of \cite[Theorem 4.26]{AlsinaBayer2004}, after replacing $\mathbb{Z}$ by $\mathbb{Z}[1/p]$. We sketch the construction of the bijection
\[
	\rho:\mathcal{E}(\mathcal{O}_{K,m}[1/p],\mathcal{O}[1/p])\simeq		\mathrm{Rep}(\mathrm{N}_{\mathcal{O}',3},-m^{2}D_{K};\mathbb{Z}[1/p]),
\]
because this will be applied later to compute examples of families of $p$-adic binary quadratic forms. 

	Let us assume first that $K=\mathbb{Q}(\sqrt{d})$ is an imaginary quadratic field such that $d\equiv2,3\,(\mathrm{mod}\,4)$, so $K$ has discriminant $D_{K}=4d$ and $\mathcal{O}_{K,m}[1/p]=\mathbb{Z}[1/p][m\sqrt{d}]$.
 
	If $\varphi:\mathcal{O}_{K,m}[1/p]\hookrightarrow\mathcal{O}[1/p]$ is an embedding, then $\omega:=\varphi(m\sqrt{d})$ is an element of $\mathcal{O}[1/p]$ with $\mathrm{Tr}_{H/\mathbb{Q}}(\omega)=\mathrm{Tr}_{K/\mathbb{Q}}(m \sqrt{d})=0$ and $\mathrm{Nm}_{H/\mathbb{Q}}(\omega)=\mathrm{Nm}_{K/\mathbb{Q}}(m\sqrt{d})=-m^{2}d$.  

	Now it is very important to observe the following fact. If $\mathcal{B}=\{1,v_{2},v_{3},v_{4}\}$ is a normalized basis for $\mathcal{O}[1/p]$, then $\mathcal{O}'[1/p]=\mathbb{Z}[1/p]+2\mathcal{O}[1/p]$ admits an integral basis $\mathcal{B}'=\{1,2v_2,2v_3,2v_4-\mathrm{Tr}(v_4)\}$ such that $2v_2,2v_3,2v_4-\mathrm{Tr}(v_4)$ are pure quaternions (cf. \cite[Definitions 1.36, 1.37]{AlsinaBayer2004}), and so a quaternion $\alpha=(a_{1},a_{2},a_{3},a_{4})_{\mathcal{B}'}$ is pure if and only if $a_{1}=0$. 

	Therefore $2\omega\in\mathcal{O}'[1/p]$, with $\mathrm{Tr}(2\omega)=0$ and $\mathrm{Nm}(2\omega)=-m^{2}4d=-m^{2}D_{K}$. After the observation above, it is clear that the integral coordinates of $2\omega$ in the basis $\mathcal{B}'$ are $(0,x,y,z)\in\mathbb{Z}[1/p]$ and give a representation $(x,y,z)\in\mathrm{Rep}(\mathrm{N}_{\mathcal{O}',3},-m^{2}D_{K};\mathbb{Z}[1/p])$. 

	Viceversa, given a representation $(x,y,z)\in\mathrm{Rep}(\mathrm{N}_{\mathcal{O}',3},-m^{2}D_{K};\mathbb{Z}[1/p])$ we can consider the quaternion $\alpha:=(0,x,y,z)_{\mathcal{B}'}\in \mathcal{O}'[1/p]$. Therefore $\alpha/2=-z\mathrm{Tr}(v_{4})/2+xv_{2}+yv_{3}+zv_{4}\in\mathcal{O}[1/p]$ is such that $\mathrm{Tr}(\alpha/2)=0$ (after what we have observed about the basis $\mathcal{B}'$) and $\mathrm{Nm}(\alpha/2)=-m^{2}d$, so we can define the embedding $\varphi:\mathcal{O}_{K,m}[1/p]\hookrightarrow\mathcal{O}[1/p]$ by setting $\varphi(m\sqrt{d}):=\alpha/2$.

	In the case $K=\mathbb{Q}(\sqrt{d})$ and $d\equiv 1\,(\mathrm{mod}\,4)$, we have that $K$ has discriminant $D_{K}=d$ and $\mathcal{O}_{K,m}[1/p]=\mathbb{Z}[1/p][(1+\sqrt{d})/2]$. We can then construct the bijection with the same reasoning. 

	Finally we can summarize the bijection $\rho^{*}$ as follows. 

	If $(x,y,z)\in\mathrm{Rep}^{*}(\mathrm{N}_{\mathcal{O}',3},-m^{2}D_{K};\mathbb{Z}[1/p])$ then we define an optimal embedding $\varphi\in\mathcal{E}^{*}(\mathcal{O}_{K,m}[1/p],\mathcal{O}[1/p])$ by
\[
\varphi(\sqrt{d}):=\left\{\begin{array}{ll}
\left(-\dfrac{z\mathrm{Tr}(v_{v_{4}})}{2m},\dfrac{x}{m},\dfrac{y}{m},\dfrac{z}{m}\right)_{\mathcal{B}},\, & \mathrm{if}\,d\equiv 2,3\,(\mathrm{mod}\,4) \\
\, & \,\\
\left(-\dfrac{z\mathrm{Tr}(v_{v_{4}})}{m},\dfrac{2x}{m},\dfrac{2y}{m},\dfrac{2z}{m}\right)_{\mathcal{B}},\, & \mathrm{if}\,d\equiv 1\,(\mathrm{mod}\,4). 
\end{array}\right.
\]

	If $\varphi\in\mathcal{E}^{*}(\mathcal{O}_{K,m}[1/p],\mathcal{O}[1/p])$ and we denote
\[
\omega:=\left\{\begin{array}{ll}
\varphi(m\sqrt{d}), & \mathrm{if}\,d\equiv 2,3\,(\mathrm{mod}\,4) \\
\, & \,\\
\varphi\left(m\dfrac{1+\sqrt{d}}{2}\right), & \mathrm{if}\,d\equiv 1\,(\mathrm{mod}\,4),
\end{array}\right.\]
	then the representation associated to $\varphi$ is $(x,y,z)\in \mathrm{Rep}^{*}(\mathrm{N}_{\mathcal{O}',3},-m^{2}D_{K};\mathbb{Z}[1/p])$ such that 
\[
\omega=\left\{\begin{array}{ll}
\left(-\dfrac{z\mathrm{Tr}(v_{4})}{2},x,y,z\right)_{\mathcal{B}},\, & \mathrm{if}\,d\equiv 2,3\,(\mathrm{mod}\,4) \\
\, & \,\\
\left(\dfrac{m-z\mathrm{Tr}(v_{4})}{2},x,y,z\right)_{\mathcal{B}},\, & \mathrm{if}\,d\equiv 1\,(\mathrm{mod}\,4).
\end{array}\right.
\]
\end{proof}

\begin{remark}
Both bijections $\rho,\rho^{*}$ are not canonical, since they depend on the choice of a normalized basis $\mathcal{B}$ for the order $\mathcal{O}[1/p].$
\end{remark}

As a direct consequence of Theorems \ref{number_embeddings} and \ref{bij_emb_rep} we have the following result. 

\begin{cor}
Let $K$ be a quadratic field of discriminant $D_{K}$, imaginary at both the primes $p$ and $\infty$, and let $\mathcal{O}_{K,m}$ be an order in $K$ of conductor $m$. Then, the following are equivalent:
\begin{enumerate}[$(i)$]
\item
$\mathcal{E}^{*}(\mathcal{O}_{K,m}[1/p],\mathcal{O}[1/p])\neq\emptyset$,
\item
$\mathcal{E}^{*}(\mathcal{O}_{K,m},\mathcal{O}_{B})\neq\emptyset$,
\item
$\mathrm{Rep}^{*}(\mathrm{N}_{\mathcal{O}',3},-m^{2}D_{K};\mathbb{Z}[1/p])\neq\emptyset$,
\item
$\mathrm{Rep}^{*}(\mathrm{N}_{\mathcal{O}'_{B},3},-m^{2}D_{K};\mathbb{Z})\neq\emptyset$,
\item
All primes $\ell\vert D_{B}$ are not split in $K$ and all primes $\ell\vert N$ are not inert in $K$.
\end{enumerate}
\end{cor}

Let $H$ be a definite quaternion algebra of discriminant $D$ and let $p$ be a prime integer such that $p\nmid D$. Let $B$ be an indefinite quaternion algebra of discriminant $Dp$. Let $\mathcal{O}_{B}\subseteq B$ and $\mathcal{O}\subseteq H$ be Eichler orders over $\mathbb{Z}$, both of level $N$, $(N,p)=1$. Put 
\[
	\begin{array}{ll}
	\mathcal{O}_{B}'=\mathbb{Z}+2\mathcal{O}_{B}, & \mathcal{O}'=\mathbb{Z}+2\mathcal{O}, \\
	\,&\,\\
  	\mathcal{O}[1/p]=\mathcal{O}\otimes\mathbb{Z}[1/p],
& \mathcal{O}'[1/p]=\mathcal{O}'\otimes\mathbb{Z}[1/p].
	\end{array}
\]
	The algebra $H$ admits a presentation $H=\left(\dfrac{\alpha,\beta}{\mathbb{Q}}\right)$ such that $\left(\dfrac{\alpha}{p}\right)=1$, which induces the $p$-adic matricial immersion 
\[
	\begin{array}{ccc}
	H & \longrightarrow & \mathrm{M}_{2}(\mathbb{Q}_{p}(\sqrt{\alpha}))=\mathrm{M}_{2}(\mathbb{Q}_{p})\\
	\, & \, & \, \\
	x+yi+zj+tk & \longmapsto  & \left(\begin{matrix}x+y\sqrt{\alpha} & z+t\sqrt{\alpha}\\ \beta(z-t\sqrt{\alpha}) & x-y\sqrt{\alpha}\end{matrix}\right),
	\end{array}
\]
and the algebra $B$ admits a presentation $B=\left(\dfrac{a,b}{\mathbb{Q}}\right)$ such that $a>0$, which induces the $\infty$-adic (or real) matricial immersion
\[
	\begin{array}{ccc}
	B & \longrightarrow & \mathrm{M}_{2}(\mathbb{Q}_{\infty}(\sqrt{a}))=\mathrm{M}_{2}(\mathbb{R})\\
	\, & \, & \, \\
	x+yi+zj+tk & \longmapsto  & \left(\begin{matrix}x+y\sqrt{a} & z+t\sqrt{a}\\ b(z-t\sqrt{a}) & x-y\sqrt{a}\end{matrix}\right).
	\end{array}
\]
	Let $K=\mathbb{Q}(\sqrt{d})$ be a quadratic field imaginary at both  the primes $\infty$ and $p$ and let $\mathcal{O}_{K,m}\subseteq K$ be an order over $\mathbb{Z}$ of conductor $m$. Put 
\[
\mathcal{O}_{K,m}[1/p]:=\mathcal{O}_{K,m}\otimes\mathbb{Z}[1/p].
\]

As usual, $B_{0}$ and $H_{0}$ denote the subset of pure quaternions in $B$ and in $H$, respectively.

We define the following sets of binary quadratic forms:
	$$\begin{array}{l}
	\mathcal{H}_{p}(\mathcal{O}'[1/p]):=
	\{f_{\Phi_{p}(\alpha)}\in\mathbb{Q}_{p}[X,Y]\mid\;\alpha\in\mathcal{O}'[1/p]\cap H_{0}\},
	\\
	\\
	\mathcal{H}_{\infty}(\mathcal{O}_{B}'):=
	\{f_{\Phi_{\infty}(\alpha)}\in\mathbb{R}[X,Y]\mid\;\alpha\in\mathcal{O}_{B}'\cap B_{0}\},
	\\
	\\
	\mathcal{H}_{p}(\mathcal{O}'[1/p],\mathcal{O}_{K,m}[1/p]):=
	\{f\in\mathcal{H}_{p}(\mathcal{O}'[1/p])\mid\;\mathrm{det}(f)=-m^{2}D_{K}\},
	\\
	\\
	\mathcal{H}_{\infty}(\mathcal{O}_{B}',\mathcal{O}_{K,m}):=
	\{f\in\mathcal{H}_{\infty}(\mathcal{O}_{B}')\mid\;\mathrm{det}(f)=-m^{2}D_{K}\}.
	\end{array}$$

	The second and the fourth sets are the same as the ones defined in \cite[Notation 4.42]{AlsinaBayer2004}. The notations $\mathcal{H}_{p}$ and $\mathcal{H}_{\infty}$ aim to recall the ones used to denote the $p$-adic and Poincar\'{e} upper half-planes, since these sets are formed by binary quadratic forms whose zeros are special points on these two upper half-planes.
	
	Note that these sets also depends on the choice of a matrix representation of the underlying quaternion algebra. This makes anyway sense, since the modular interpretation of the associated Shimura curve also depends on this datum (cf. \cite{BayerGuardia2005})

\begin{prop}\label{bij_forms_embeddings}
	The following are well-defined and bijective maps:
	\begin{enumerate}[$(i)$]
		\item 
		$\mathfrak{f}_{\infty}:\mathcal{E}(\mathcal{O}_{K,m},\mathcal{O}_{B})\rightarrow \mathcal{H}_{\infty}(\mathcal{O}'_{B},\mathcal{O}_{K,m}),$ where $\mathfrak{f}_{\infty}(\varphi) := f_{\Phi_{\infty}(\varphi(m\sqrt{D_{K}}))}.$
		\item
		$\mathfrak{f}_{p}:\mathcal{E}(\mathcal{O}_{K,m}[1/p],\mathcal{O}[1/p])\rightarrow\mathcal{H}_{p}(\mathcal{O}'[1/p],\mathcal{O}_{K,m}[1/p]),$ where 
		$$\mathfrak{f}_{p}(\varphi) := f_{\Phi_{p}(\varphi(m\sqrt{D_{K}}))},$$
	for every $\varphi\in \mathcal{E}(\mathcal{O}_{K,m}[1/p],\mathcal{O}[1/p]).$
	\end{enumerate}
\end{prop}

\begin{proof}
	For the proof of $(i)$ see \cite[Theorem 4.53]{AlsinaBayer2004}.
	For the proof of $(ii)$, let us observe that the element $e:=\varphi(m\sqrt{D_{K}})$ belongs to $\mathcal{O}'[1/p]$ and is such that $\mathrm{Nm}_{H/\mathbb{Q}}(e) = -m^{2}D_{K}$ and $\mathrm{Tr}_{H/\mathbb{Q}}(e) = 0$ (cf. also the proof of Theorem \ref{bij_emb_rep}).
	
	Therefore $f_{\Phi_{p}(e)}\in \mathcal{H}_{p}(\mathcal{O}'[1/p])$ and $\mathrm{det}(f_{\Phi_{p}(e)}) = -m^{2}D_{K}$. The map $\mathfrak{f}_{p}$ is then well-defined and proving that is a bijection is straightforward. 
\end{proof}

\begin{remark}
Note that, as in Theorem \ref{bij_emb_rep}, bijections $\mathfrak{f}_{\infty}$ and $\mathfrak{f}_{p}$ are not cano\-nical as they depend on the choice of a matricial representation $\Phi_{\infty}$ and $\Phi_{p}$, respectively.
\end{remark}

\begin{defn}
	We define the following sets of ``primi\-tive'' binary quadratic forms.
	\begin{enumerate}[$(a)$]
		\item
		$\mathcal{H}_{\infty}^{*}(\mathcal{O}_{B}',\mathcal{O}_{K,m}):=\mathfrak{f}_{\infty}(\mathcal{E}^{*}(\mathcal{O}_{K,m},\mathcal{O}_{B})),$
		\item
		$\mathcal{H}_{p}^{*}(\mathcal{O}'[1/p],\mathcal{O}_{K,m}[1/p]):=\mathfrak{f}_{p}(\mathcal{E}^{*}(\mathcal{O}_{K,m}[1/p],\mathcal{O}[1/p]))$.
	\end{enumerate}
	We refer by {$(\mathcal{O}_{B},\mathcal{O}_{K,m})$-primitive quadratic form} to a binary quadratic form belonging to the set $\mathcal{H}_{\infty}^{*}(\mathcal{O}_{B}'\mathcal{O}_{K,m}),$ and by $(\mathcal{O}[1/p],\mathcal{O}_{K,m}[1/p])$-primitive quadratic form to a binary quadratic form belonging to $\mathcal{H}_{p}^{*}(\mathcal{O}'[1/p],\mathcal{O}_{K,m}[1/p])$.
\end{defn}

Combining bijections defined in Theorem \ref{bij_emb_rep} and Proposition \ref{bij_forms_embeddings} we obtain the following result, which mirrors \cite[Corollary 4.5]{AlsinaBayer2004}.

\begin{cor}
There exist bijective maps:
	\begin{enumerate}[$(i)$]
		\item
		$\mathcal{E}(\mathcal{O}_{K,m},\mathcal{O}_{B})\simeq
		\mathrm{Rep}(\mathrm{N}_{\mathcal{O}'_{B},3},-m^{2}D_{K};\mathbb{Z})		\simeq
		\mathcal{H}_{\infty}(\mathcal{O}'_{B},\mathcal{O}_{K,m}),$
		\item
		$\mathcal{E}(\mathcal{O}_{K,m}[1/p],\mathcal{O}[1/p])\simeq\\
		\simeq
		\mathrm{Rep}(\mathrm{N}_{\mathcal{O}',3},-m^{2}D_{K};\mathbb{Z}[1/p])		\simeq\mathcal{H}_{p}(\mathcal{O}'[1/p],\mathcal{O}_{K,m}[1/p]),$
	\end{enumerate}
	which restrict to bijections of sets of optimal embeddings, primitive representations and primitive forms:
	\begin{enumerate}[$(i)$]
		\item
		$\mathcal{E}^{*}(\mathcal{O}_{K,m},\mathcal{O}_{B})\simeq\mathrm{Rep}^{*}(\mathrm{N}_{\mathcal{O}'_{B},3},-m^{2}D_{K};\mathbb{Z})\simeq
\mathcal{H}_{\infty}^{*}(\mathcal{O}'_{B},\mathcal{O}_{K,m})$,
		\item
		$\mathcal{E}^{*}(\mathcal{O}_{K,m}[1/p],\mathcal{O}[1/p])\simeq\\
		\simeq\mathrm{Rep}^{*}(\mathrm{N}_{\mathcal{O}',3},-m^{2}D_{K};\mathbb{Z}[1/p])\simeq
		\mathcal{H}_{p}^{*}(\mathcal{O}'[1/p],\mathcal{O}_{K,m}[1/p])$.
	\end{enumerate}
\end{cor}

\section{Class numbers of $p$-adic binary quadratic forms and CM points}
	Now we will rephrase the definition of complex multiplication parameter arising from the complex uniformization of a Shimura curve $X(Dp,N)$ (cf. \cite{AlsinaBayer2004}) by using a notation which is more enlightening in this context. In fact, this new formulation will allow us to establish the analogous definition for the $p$-adic uniformization of a Shimura curve, thanks to the interchange of the local invariants $p$ and $\infty$. 

	Let $B$ be an indefinite quaternion algebra of discriminant $Dp$  and let $\mathcal{O}_{B}\subseteq B$ be an Eichler order over $\mathbb{Z}$ of level $N$. Let $H$ be the definite quaternion algebra of discriminant $D$ and let $\mathcal{O}_{B}[1/p]\subseteq B$ be an Eichler order over $\mathbb{Z}[1/p]$ of conductor $N$. Fix two matricial immersions:
\[
\begin{array}{lccr}
\Phi_{\infty}:B\hookrightarrow \mathrm{M}_{2}(\mathbb{Q}_{\infty}), &\:&\:& \Phi_{p}:H\hookrightarrow \mathrm{M}_{2}(\mathbb{Q}_{p}).
\end{array}
\]

	Therefore we can consider the following discrete subgroups:
\begin{enumerate}[(a)]
\item
$\Gamma_{\infty}:=\Phi_{\infty}(\mathcal{O}_{B}^{*})/\mathbb{Z}^{*}\subseteq\mathrm{PGL}_{2}(\mathbb{Q}_{\infty})$,
\item
$\Gamma_{p}:=\Phi_{p}(\mathcal{O}[1/p]^{*})/\mathbb{Z}[1/p]^{*}\subseteq\mathrm{PGL}_{2}(\mathbb{Q}_{p})$.
\item
$\Gamma_{\infty,+}:=\Phi_{\infty}(\mathcal{O}_{B,+}^{*})/\mathbb{Z}^{*}\subseteq\mathrm{PGL}_{2}(\mathbb{Q}_{\infty})$,
\item
$\Gamma_{p,+}:=\Phi_{p}(\mathcal{O}[1/p]_{+}^{*})/\mathbb{Z}[1/p]^{*}\subseteq\mathrm{PGL}_{2}(\mathbb{Q}_{p})$.
\end{enumerate}

\begin{defn}
	Let $K=\mathbb{Q}(\sqrt{d})$ be a quadratic field, imaginary at both  the primes $\infty$ and $p$, and let $\mathcal{O}_{K,m}$ be an order of $K$ of conductor $m$. Put $\mathcal{O}_{K,m}[1/p]:=\mathcal{O}_{K,m}\otimes_{\mathbb{Z}}\mathbb{Z}[1/p]$.

	\begin{enumerate}[$(a)$]
		\item
		We say that a point $\tau\in\Gamma_{\infty,+}\backslash\mathcal{H}$ is an {$\infty$-imaginary multi\-plication parameter by $\mathcal{O}_{K,m}$} (also a complex multiplication parameter) if $\tau\in\mathcal{H}$ is a fixed point for an optimal $\mathbb{Z}$-embedding $\varphi:\mathcal{O}_{K,m}\hookrightarrow\mathcal{O}_{B}$. 

		We denote by 
		$$\mathrm{CM}_{\infty}(Dp,N,d,m)$$ the subset in $\Gamma_{\infty,+}\backslash\mathcal{H}$ of complex multi\-plication parameters by $\mathcal{O}_{K,m}$, and its cardinality by $\mathrm{cm}_{\infty}(Dp,N,d,m)$.
		\item
		We say that a point $\tau\in\Gamma_{p,+}\backslash\mathcal{H}_{p}(\mathbb{Q}_{p^{2}})$ is a {$p$-imaginary multi\-plication parameter by $\mathcal{O}_{K,m}[1/p]$} if $\tau\in\mathcal{H}_{p}(\mathbb{Q}_{p^{2}})$ is a fixed point for an optimal $\mathbb{Z}[1/p]$-embedding $\varphi:\mathcal{O}_{K,m}[1/p]\hookrightarrow\mathcal{O}[1/p]$. 

		We denote by 
		$$\mathrm{CM}_{p}(D,N,d,m)$$ the subset in $\Gamma_{p,+}\backslash\mathcal{H}_{p}(\mathbb{Q}_{p^{2}})$ of $p$-imaginary multiplication parameters by the order $\mathcal{O}_{K,m}[1/p],$ and its cardinality by $\mathrm{cm}_{p}(D,N,d,m)$.
		\end{enumerate}
\end{defn}

	As observed in Remark \ref{expression_fixed_points}, an embedding $\varphi: K\hookrightarrow H$ has two Galois conjugated fixed points $z,\overline{z}\in\mathbb{Q}_{p^{2}}\smallsetminus\mathbb{Q}_{p}$, the Galois conjugation being in this case the one coming from the Galois group of the quadratic extension $\mathbb{Q}_{p^{2}}|\mathbb{Q}_{p}$. 

\begin{defn}
	Let $K$ be a quadratic field, $K=\mathbb{Q}(\sqrt{d})$ with $d$ a square-free integer. For every embedding $\varphi\in\mathcal{E}(K,H)$ there exists a unique embedding $\psi\in\mathcal{E}(K,H)$ such that:
\[
\varphi(\sqrt{d})=\alpha\iff\psi(-\sqrt{d})=-\alpha .
\]
	The embedding $\psi$ is exactly the embedding $\varphi$ composed with the Galois conjugation of the Galois group $\mathrm{Gal}(K/\mathbb{Q})$ and so it has the same fixed points of $\varphi$. The embedding $\psi$ is called {the conjugated embedding of $\varphi$} and it is denoted by $\overline{\varphi}$.

	Moreover, since $\overline{\varphi}(\mathcal{O}_{K,m})=\varphi(\mathcal{O}_{K,m})$ and $\overline{\varphi}(K)=\varphi(K)$, we have that 
\[
\varphi\in\mathcal{E}^{*}(\mathcal{O}_{K,m}[1/p],\mathcal{O}_{B}[1/p])\iff \overline{\varphi}\in\mathcal{E}^{*}(\mathcal{O}_{K,m}[1/p],\mathcal{O}_{B}[1/p]).
\]
\end{defn}

	The relation between conjugated embeddings and their fixed points is explained in the following result.

\begin{lemma}\label{lemma_embeddings}
	Let $\varphi,\varphi'\in\mathcal{E}^{*}(\mathcal{O}_{K,m}[1/p],\mathcal{O}[1/p])$ be optimal embeddings such that 
	\begin{enumerate}[$(i)$]
		\item
		$\{z=z(\varphi),\overline{z}=\overline{z(\varphi)}\}$ are the fixed points of $\varphi$,
		\item
		$\{z'=z(\varphi'),\overline{z'}=\overline{z(\varphi')}\}$ are the fixed point of $\varphi'$.
	\end{enumerate}
	Then $z'$ is $\Gamma_{p,+}$-equivalent to $z$ or to $\overline{z}$ if and only if $\varphi'$ is $\mathcal{O}[1/p]^{*}_{+}$-equivalent to $\varphi$ or to $\overline{\varphi}$.
\end{lemma}

\begin{proof}
	The proof follows the same idea as in \cite[Proposition 6.11]{AlsinaBayer2004} with the difference that in our case the transformations have two fixed and Galois conjugated points. 

First observe that:
	\begin{enumerate}[(a)]
		\item
		$z(\overline{\varphi})=\overline{z(\varphi)}$,
		\item
		$z'=Pz\iff \overline{z'}=P\overline{z}$,\, for every $P\in\mathrm{PGL}_{2}(\mathbb{Q}_{p})$.
	\end{enumerate}

	Let us assume for example that $z=Pz'$, for some $P\in\Gamma_{p}$, and let $\sigma\in\mathcal{O}[1/p]^{*}_{+}/\mathbb{Z}[1/p]$ be such that $\Phi_{p}(\sigma)=P$.

	Let $\alpha\in K^{*}$, $\mathrm{Tr}(\alpha)=0$, and put 
	\[\begin{array}{lcr}
	\gamma:=\Phi_{p}(\varphi(\alpha)), & \gamma':=\Phi_{p}(\varphi'(\alpha)), & \gamma'':=\Phi_{p}(\sigma^{-1}\varphi(\alpha)\sigma).
	\end{array}\]
	Therefore $\gamma''$ has fixed points $\{\overline{z}',z'\}$, the same as $\gamma'$ so by Lemma \ref{lemma} $\gamma'=\lambda\gamma''+\mu\mathrm{I}_{2}$, for some $\lambda,\mu\in\mathbb{Q}_{p}$, $\lambda\neq 0$. 
		A computation of the trace and of the determinant of $\gamma'$ and $\gamma''$ gives the equality $\mathrm{Det}(\gamma')=\lambda^{2}\mathrm{Det}(\gamma'')$ and so $\lambda=\pm 1$. 
		Hence we have that $\varphi'(\alpha)=\sigma^{-1}(\pm\varphi)(\alpha)\sigma$ for every $\alpha\in K^{*}$ of null-trace, so $\varphi$ is $\mathcal{O}[1/p]_{+}^{*}$-equivalent to $\varphi$ or to $\overline{\varphi}$.
		If we assume that $\overline{z}=Pz'$, with the same reasoning we obtain that $\varphi$ has to be $\mathcal{O}[1/p]_{+}^{*}$-equivalent to $\overline{\varphi}$ or to $\varphi$.

	For the viceversa let us assume that $\varphi'$ is $\mathcal{O}[1/p]_{+}^{*}$-equivalent to $\varphi$, i.e. there exists a $\sigma\in\mathcal{O}[1/p]_{+}^{*}$ such that $\varphi'=\sigma^{1}\varphi\sigma$. For every $\alpha\in K$ put 
	\[\begin{array}{lcr}
	P:=\Phi_{p}(\sigma) & \gamma:=\Phi_{p}(\varphi(\alpha)) & \gamma':=\Phi_{p}(\varphi'(\alpha)).
	\end{array}\]
	Therefore $\gamma'=P^{-1}\gamma P$ and a simple calculation yields that $\gamma'$ fixes $P^{-1}z$ and $P^{-1}\overline{z}$. Since $\gamma$ has as unique fixed points $z'$ and $\overline{z}'$, then we have that $P^{-1}z=z'$ or $P^{-1}z=\overline{z'}$. Again, if we assume that $\varphi'$ is $\mathcal{O}[1/p]_{+}^{*}$-equivalent to $\overline{\varphi}$, we prove that $z$ has to be $\Gamma_{p,+}$-equivalent to $\overline{z'}$ or $z'$.
\end{proof}

	\begin{lemma}\label{equiv_fixed_points}
	Let $\varphi\in\mathcal{E}^{*}(\mathcal{O}_{K,m}[1/p],\mathcal{O}[1/p])$ be an optimal embedding with fixed points $\{z,\overline{z}\}$. 

	Then $z$ and $\overline{z}$ are never $\mathrm{PGL}_{2}(\mathbb{Q}_{p})$-equivalent and they represent different classes in $\Gamma_{p,+}\backslash\mathcal{H}_{p}(\mathbb{Q}_{p^{2}})$.
	\end{lemma}

\begin{proof} 
	If $\gamma\in \mathrm{PGL}_{2}(\mathbb{Q}_{p})$ is a transformation such that $\gamma\cdot z=\overline{z}$, then $\gamma^{2}$ fixes $z$ and $\overline{z}$, so it has to be an elliptic transformation. 

	Now, on the one hand, a calculation shows that  $\gamma\cdot z=\overline{z}$ implies $\mathrm{Tr}(\gamma)=0$. 
	On the other hand, $\mathrm{Tr}(\gamma^{2})=\mathrm{Tr}(\gamma)^{2}-2\mathrm{Det}(\gamma)=-2\mathrm{Det}(\gamma)$. 

	Finally $\mathrm{Tr}(\gamma^{2})^{2}-4\mathrm{Det}(\gamma^{2})=0$, and so by Proposition \ref{caracterizacion_transf_formas} we find that $\gamma^{2}$ is a parabolic transformation, which is an absurd. 

	Let $\mathcal{F}$ be a fundamental domain in $\mathcal{H}_{p}(\mathbb{Q}_{p^{2}})$ for the action of $\Gamma_{p,+}$, so there exists $\gamma\in \Gamma_{p,+}$ such that $z':=\gamma\cdot z\in\mathcal{F}$. 

	Therefore $\gamma\cdot \overline{z}=\overline{\gamma\cdot z}$, since $\gamma\in\mathrm{PGL}_{2}(\mathbb{Q}_{p})$ and so its coefficients are auto-conjugated. 
	Moreover $\overline{\gamma\cdot z}\in\mathcal{F}$, since $|z'-a|=|\overline{z'}-a|$ for every $a\in\mathbb{Q}_{p}$, by the definition of the $p$-adic absolute value on $\mathbb{Q}_{p^{2}}$.   
	This concludes the proof.
\end{proof}

\begin{cor}\label{equiv_embeddings}
	Let $\varphi\in\mathcal{E}(K,H)$ be an embedding. Then $\varphi$ is not $\mathcal{O}[1/p]^{*}_{+}$-equivalent to $\overline{\varphi}$.
\end{cor}

\begin{proof}
	Let us suppose that the two embeddings $\varphi,\overline{\varphi}$ are $\mathcal{O}[1/p]^{*}_{+}$-equivalent, and let $z,\overline{z}$ be their fixed points. 
	By Lemma \ref{lemma_embeddings} we know then that $z$ has to be equivalent to $z$ or to $\overline{z}$, and Lemma \ref{equiv_fixed_points} gives that $z$ can only be $\Gamma_{p,+}$-equivalent to $z$. 
	Let us define, as usual, $P :=\Phi_{p}(\sigma)$ where $\sigma\in\mathcal{O}[1/p]^{*}_{+}$ is such that $\sigma\varphi\sigma^-1 = \overline{\varphi}$, and let us take the field $K$ to be $\mathbb{Q}(\sqrt{d}).$ 
	Since $P$ has $z,\overline{z}$ as fixed points, by Lemma \ref{lemma} we find that $P = \lambda\mathrm{I}_2 + \mu \Phi_p(\sqrt{d})$ for some $\lambda,\mu\in\mathbb{Q}_p,\,\lambda \ne 0$.
	An easy calculation now shows hat $\varphi(\sqrt{d}) = (-\mu/\lambda) d$, which is an absurd.
	\end{proof}

	This last result allows us to prove the following theorem, where the number of classes of optimal embeddings is related to the number of $p$-imaginary multiplication parameters (cf. with its archimedean analog: \cite[Theorem 6.13]{AlsinaBayer2004}).

\begin{teo}\label{numb_p-adic_cm}
	The set $\mathrm{CM}_{p}(D,N,d,m)$ is finite and has cardinality 
\[
\mathrm{cm}_{p}(D,N,d,m)=\nu(D,N,d,m; \mathcal{O}[1/p]^{*}_{+}).
\]
	Moreover this cardinality is an even integer.
\end{teo}

\begin{proof}
	By Lemma \ref{lemma_embeddings}, Lemma \ref{equiv_fixed_points} and Corollary \ref{equiv_embeddings} we know that assigning to an optimal embedding $\varphi\in\mathcal{E}^{*}(\mathcal{O}_{K,m}[1/p],\mathcal{O}[1/p])$ one of its two fixed points $\{z(\varphi),\overline{z(\varphi)}\}$ induces a well-defined bijection between the sets 
	\[
	\mathcal{O}[1/p]^{*}_{+}\backslash\mathcal{E}^{*}(\mathcal{O}_{K,m}[1/p],\mathcal{O}[1/p])\simeq\mathrm{CM}_{p}(D,N,d,m).
	\]
	The fact that the cardinality is even is clear from the bijection, because for every point/embedding we have a different class given by the Galois conjugated.
\end{proof}

	Note that the archimedean analog of this result, which is \cite[Theorem 6.13]{AlsinaBayer2004}, is ``half of Theorem \ref{numb_p-adic_cm}'' since the cardinality of the set of $\infty$-imaginary multiplication points is given by the formula:
	\[
	\mathrm{cm}_{\infty}(Dp,N,d,m)=\frac{1}{2}\nu(Dp,N,d,m; \mathcal{O}_{B,+}^{*}).
	\] 

\begin{cor} \label{main_coro}
	The sets of points $\mathrm{CM}_{\infty}(Dp,N,d,m)$ and $\mathrm{CM}_{p}(D,N,d,m)$ have the same cardinality, i.e. 
\[
\mathrm{cm}_{p}(D_{B},N,d,m)=\mathrm{cm}_{\infty}(D_{H},N,d,m).
\]
\end{cor}

\begin{proof}
	This is a consequence of Theorem \ref{numb_p-adic_cm}, Theorem \ref{number_embeddings} and \cite[Theorem 6.13]{AlsinaBayer2004}.
\end{proof}

	Now, we can also define the cardinalities of certain sets of binary quadratic forms classes and we can relate them with the number of CM points.
	\begin{enumerate}[(a)]
		\item
		$\mathrm{h}_{p}(D,N,d,m):=\#\mathcal{H}_{p}^{*}(\mathcal{O}'[1/p],\mathcal{O}_{K,m})/\Phi_{p}(\mathcal{O}[1/p]^{*}/\mathbb{Z}[1/p]^{*})$,
		\item
		$\mathrm{h}_{\infty}(Dp,N,d,m):=\#\mathcal{H}_{\infty}^{*}(\mathcal{O}_{B}',\mathcal{O}_{K,m})/\Phi_{\infty}(\mathcal{O}_{B}^{*}/\mathbb{Z}^{*})$.
	\end{enumerate}

\begin{cor}
	We have the following equalities between cardinalities:
	\begin{enumerate}[$(i)$]
		\item
		$\mathrm{h}_{p}(D_{B},N,d,m)=\nu(D_{B},N,d,m; \mathcal{O}_{B}[1/p]^{*})=\frac{1}{2}\mathrm{cm}_{p}(D_{B},N,d,m)$,
		\item
		$\mathrm{h}_{\infty}(D_{H},N,d,m)=\nu(D_{H},N,d,m; \mathcal{O}_{H}^{*})=\mathrm{cm}_{\infty}(D_{H},N,d,m)$. 
	\end{enumerate}
\end{cor}

\section{Computations in a family of Shimura curves of discriminant $2p$}

	In the following example we compute families of binary quadratic forms arising from the $p$-adic uniformization of Shimura curves of discriminant $2p$ and level $N = 1$, when $p \equiv 1 \, (\mathrm{mod}\, 4)$.
	We refer the reader to \cite{AmorosMilione2017} for an explicit description of the $p$-adic uniformization this and other families of Shimura curves.
\[\]
	$H=\left(\dfrac{-1,-1}{\mathbb{Q}}\right)$,\, $D=2$, \,$p\equiv 1\,(\mathrm{mod}\,4)$,
\[\]
	$\mathcal{O}:=\mathbb{Z}[1,i,j,\rho]$, where $\rho:=(1+i+j+k)/2$, 
\[\]
	$\mathcal{O}':=\mathbb{Z}+2\mathcal{O}=\langle1,2i,2j,2\rho-\mathrm{Tr}(\rho)\rangle=\langle1,2i,2j,\rho'\rangle$,  where $\rho':=i+j+k$,
\[\]
	$\mathcal{O}'[1/p]:=\mathcal{O}'\otimes_{\mathbb{Z}}\mathbb{Z}[1/p]$,
\[\]
	$\mathrm{N}_{H,4}(X,Y,Z,Y)=X^{2}+Y^{2}+Z^{2}+T^{2}$, 
\[\]
	$\mathrm{N}_{\mathcal{O}',4}(X,Y,Z,Y)=X^{2}+(2Y+T)^{2}+(2Z+T)^{2}+T^{2}$.
\[\]
	$\mathrm{N}_{\mathcal{O}',3}(X,Y,Z)=(2X+Z)^{2}+(2Y+Z)^{2}+Z^{2}$,  
\[\]
	For every $\alpha=(2y+t)i+(2z+t)j+tk\in\mathcal{O}'[1/p]\cap B_{0}$ we have

\[
\Phi_{p}(\alpha)=\left(\begin{matrix}(2y+t)\sqrt{-1} & 2z+t+t\sqrt{-1} \\
-(2z+t-t\sqrt{-1}) & -(2y+t)\sqrt{-1}\end{matrix}\right)\in\mathrm{M}_{2}(\mathbb{Q}_{p}),
\]
\begin{multline*}
f_{\Phi_{p}(\alpha)}= [\,-(2z+t)+t\sqrt{-1},\,  -2(2y+t)\sqrt{-1},\,  -(2z+t)-t\sqrt{-1}x \,]\\
\in\mathbb{Z}[1/p][\sqrt{-1}][X,Y].
\end{multline*}

If $a:=-(2z+t), b:=t, c:=-(2y+t)$ then

\[\mbox{
$a,b,c\in\mathbb{Z}[1/p]$ such that $b+c\equiv a+b\equiv 0\,(\mathrm{mod}\,2).$ 
}\]

	Let $\mathcal{O}_{K,m}$ be an order of conductor $m$ in a quadratic field $K$ of discri\-minant $D_{K}$, where $K =\mathbb{Q}(\sqrt{d})$ for some  $d<0$ and $\left(\frac{d}{p}\right)=-1$.

	Since $\mathrm{det}(f_{\Phi_{p}(\alpha)})=\mathrm{Nm}(\alpha)=a^{2}+b^{2}+c^{2}$, we have

\begin{multline*}
	\mathcal{H}_{p}(\mathcal{O}'[1/p],\mathcal{O}_{K,m}[1/p])=
	\{[\,a+b\sqrt{-1},\,2c\sqrt{-1},\,a-b\sqrt{-1}\,]\mid \\
b+c\equiv a+b\equiv 0\,(\mathrm{mod}\,2), \;a^{2}+b^{2}+c^{2}=-m^{2}D_{K}\}.
\end{multline*}

	We now compute the subset of $(\mathcal{O}[1/p],\mathcal{O}_{K,m}[1/p])$-primitive forms.

	As shown in the proof of Theorem \ref{bij_emb_rep}, to every primitive representation 

	$(x,y,z)\in\mathrm{Rep}^{*}(\mathrm{N}_{\mathcal{O}',3},-m^{2}D_{K};\mathbb{Z}[1/p])$, corresponds a pure quaternion 
\[
\alpha=(2x+z)i+(2y+z)j+zk\in B
\]
of norm $\mathrm{Nm}(\alpha)=\mathrm{N}_{\mathcal{O},3}(x,y,z)=-m^{2}D_{K}$. Therefore the associated binary quadratic form 
\[
f_{\Phi_{p}(\alpha)}=[\,-(2y+z)+z\sqrt{-1},\,-2(2x+z)\sqrt{-1},\,-(2y+z+z\sqrt{-1})\,],
\]
has determinant $\mathrm{det}(f_{\Phi_{p}(\alpha)})=-m^{2}D_{K}$.

	Finally, putting 
\[
\begin{array}{lcr}
a:=-(2y+z),& b:=z, & c:=-(2x+z),
\end{array}
\]
we have that
\begin{multline*}
	\mathcal{H}^{*}_{p}(\mathcal{O}'[1/p],\mathcal{O}_{K,m}[1/p])=\\
=\{[\, a+b\sqrt{-1},\, 2c\sqrt{-1},\, a-b\sqrt{-1}\,] \in\mathcal{H}_{p}(\mathcal{O}'[1/p],\mathcal{O}_{K,m}[1/p])\mid\,\\
\left((a+b)/2, (c+b)/2, b\right)=\mathbb{Z}[1/p]\}.
\end{multline*}

	Computing the zeros of these forms, we find that the set of $p$-imaginary multiplication parameters $\mathrm{CM}_{p}(2,1,d,m)$ is the the set of $\Gamma_{p}(2,1)$-equivalence classes of the set of points
\begin{multline*}
	\{	
\frac{(2x+z)\sqrt{-1}\pm m\sqrt{D_{K}}}{-(2y+z)+z\sqrt{-1}}\in\mathcal{H}_{p}(\mathbb{Q}_{p^{2}})\mid\,\\
	(x,y,z)\in\mathrm{Rep}^{*}(\mathrm{N}_{\mathcal{O}',3},-m^{2}D_{K};\mathbb{Z}[1/p])
	\}=
\end{multline*}
\begin{multline*}	
	=\{
	\frac{-c\sqrt{-1}\pm m\sqrt{D_{K}}}{a+b\sqrt{-1}}\in\mathcal{H}_{p}(\mathbb{Q}_{p^{2}})\mid (a,b,c)\in\mathbb{Z}[1/p]^{3},\,\\
	a^{2}+b^{2}+c^{2}=-m^{2}D_{K},\; \left((a+b)/2, (c+b)/2, b \right)=\mathbb{Z}[1/p] 
\}.
\end{multline*}

\begin{example}\label{Ex}
Let us take $p = 5$. Since $\left(\frac{-1}{5}\right) = 1,$ the algebraic number $i :=\sqrt{-1}$ belongs to $\mathbb{Q}_{5}$. In particular, we can take

$$i := 2+1\cdot 5+1\cdot 5^{2}+\mathcal{O}(5^{3}).$$
If we also fix the values $(d,m) = (-2, 1)$ , by Corollary \ref{main_coro} and  \cite[Table A.35]{AlsinaBayer2004}, we find that 

$$\mathrm{cm}_{5}(2,1,-2,1) = \mathrm{cm}_{\infty}(2\cdot 5, 1, -2, 1) = 2.$$ 

Therefore any representation $(a,b,c)\in\mathrm{Rep}(-8, X^2+Y^2+Z^2, \mathbb{Z}[1/5]),$ satisfying $\left((a+b)/2, (c+b)/2, b \right)=\mathbb{Z}[1/p]$, gives rise to a couple of Galois conjugated parameters in $\mathbb{Q}_{5}(\sqrt{-2})\simeq \mathbb{Q}_{5^{2}}$. By Lemma \ref{equiv_fixed_points} we know that these two parameters are not $\Gamma_{5,+}$-equivalent. Hence

$$\mathrm{CM}_{p}(2, 1, -2, 1) = \{[1 +i\sqrt{-2}], [1-i\sqrt{-2}]\}\subseteq\Gamma_{5,+}\backslash\mathcal{H}_{5}(\mathbb{Q}_{5}(\sqrt{-2})).$$ 
\end{example}

\begin{remark}
The conjugated points from the previous example are associated to the representation $(a,b,c) = (0,2,-2)$ or, equivalently, to the primitive representation 
$$(x,y,z) = (0, -1, 2)\in\mathrm{Rep}^{*}(\mathrm{N}_{\mathcal{O}',3}, -8 ;\mathbb{Z}[1/5]).$$

There are other primitive representations in 
$$\mathrm{Rep}(X^2+Y^2+Z^2, -8;\mathbb{Z}[1/5]),$$ e.g., $(a,b,c) = (0,2,2)$. This last one, for example, gives rise to the  Galois conjugated points $-1\pm i\sqrt{-2}$, which are non $\Gamma_{5,+}$-equivalent, again by Lemma \ref{equiv_fixed_points}. Nevertheless, each of them is equivalent to one of the previous points $1 \pm i\sqrt{-2}$, by the transformation represented by $\left(\begin{smallmatrix}i & 0 \\ 0 & i\end{smallmatrix}\right)$ and induced by the unit quaternion $i \in\mathcal{O}^{\times}$ (cf. \cite[Proposition 4.4 ($ii$)]{AmorosMilione2017}). 
\end{remark}

\begin{remark}
The same reasoning as in Example \ref{Ex} can be easily applied for computing the set 
$\mathrm{CM}_{p}(D,N, d, m)$
whenever its cardinality is $2$.
When 
$$\mathrm{cm}_{p}(D,N,d,m) =\mathrm{cm}_{\infty}(Dp, N, d,  m) > 2,$$ 
we need to be able to decide whether two points are $\Gamma_{p,+}-$equivalent or not. This can be done by first computing a fundamental domain for the action of $\Gamma_{p,+}$, then applying a \emph{reduction point algorithm} associated to the fundamental domain, which gives a representative of each point inside the fundamental domain (cf. \cite{BayerRemon2014} for the arquimedean analogous of such an algorithm). 

Anyway, this is not an easy problem as there is no general method for computing $p$-adic fundamental domains with respect to the cocompact groups $\Gamma_{p,+}$. In \cite{AmorosMilione2017} we computed fundamental domains for certain infinite families of Shimura curves in relation to Mumford curves covering them. It will be the object of a future work to making explicit a reduction point algorithm for those fundamental domains.
\end{remark}




\normalsize
\baselineskip=17pt


\bibliography{BibliografiaPM}{}

\begin{bibdiv}
\begin{biblist}

\bib{AlsinaBayer2004}{book}{
      author={Alsina, M.},
      author={Bayer, P.},
       title={{Q}uaternion orders, quadratic forms, and {S}himura curves},
      series={CRM Monograph Series},
   publisher={AMS},
        date={2004},
      volume={22},
}

\bib{AmorosMilione2017}{article}{
      author={Amor{\'o}s, L.},
      author={Milione, P.},
       title={Mumford curves covering $p$-adic {S}himura curves and their
  fundamental domains},
     journal={to appear in Trans. of the AMS},
}

\bib{BayerGuardia2005}{article}{
      author={Bayer, P.},
      author={Gu\`{a}rdia, J.},
       title={On equations defining fake elliptic curves},
        date={2005},
     journal={Journal de Th\'{e}orie des nombres de Bordeaux},
      volume={17},
      number={1},
       pages={57\ndash 67},
}

\bib{BayerRemon2014}{article}{
      author={Bayer, P.},
      author={Rem{\'o}n, D.},
       title={A reduction point algorithm for cocompact {F}uchsian group and
  applications},
        date={2014},
     journal={Adv. Math. Commun.},
      volume={8},
       pages={223\ndash 239},
}

\bib{BayerTravesa2007}{article}{
      author={Bayer, P.},
      author={Travesa, A.},
       title={Uniformizing functions for certain {S}himura curves, in the case
  ${D}=6$},
        date={2007},
     journal={Acta Arithmetica},
      volume={126},
       pages={315\ndash 339},
}

\bib{BayerTravesa2008}{article}{
      author={Bayer, P.},
      author={Travesa, A.},
       title={On local constants associated to arithmetical automorphic
  functions},
        date={2008},
     journal={Pure and Applied Mathematics Quaterly. Special Issue: In honor of
  Jean-Pierre Serre},
      number={1},
       pages={1107\ndash 1132},
}

\bib{BoutotCarayol1991}{article}{
      author={Boutot, J.-F.},
      author={Carayol, H.},
       title={Uniformisation $p$-adique des courbes de {S}himura: les
  th\'{e}or\`{e}mes de \v{C}erednik et de {D}rinfeld},
        date={1991},
     journal={Ast\'{e}risque},
      number={196-197},
       pages={45\ndash 158},
        note={Courbes modulaires et courbes de Shimura (Orsay, 1987/1988)},
}

\bib{BullBook}{book}{
      author={Buell, D.~A.},
       title={Binary quadratic forms, classical theory and applications},
   publisher={Springer},
        date={1989},
}

\bib{CoxBook}{book}{
      author={Cox, D.~A.},
       title={Primes of the form $x^2+ny^2$: {F}ermat, {C}lass {F}ield
  {T}heory, and {C}omplex {M}ultiplication},
      series={A Wiley Series of Texts, Monographs, and Tracts},
   publisher={John Wiley and Sons},
        date={1989},
}

\bib{DarmonBook}{book}{
      author={Darmon, H.},
       title={Rational points on modular elliptic curves},
   publisher={AMS},
        date={2003},
}

\bib{Eichler1955_crelle}{article}{
      author={Eichler, M.},
       title={{Z}ur {Z}ahlentheorie der {Q}uaternionen-{A}lgebren},
        date={1955},
     journal={J. Reine Angew. Math.},
      number={195},
       pages={127\ndash 151},
}

\bib{Gerritzen_vanderPut1980}{book}{
      author={Gerritzen, L.},
      author={van Der~Put, M.},
       title={{S}chottky groups and {M}umford curves},
      series={Lecture Notes in Mathematics},
   publisher={Springer},
        date={1980},
      volume={817},
        ISBN={9780387102290},
         url={http://books.google.es/books?id=NzvvAAAAMAAJ},
}

\bib{Greenberg2009}{article}{
      author={Greenberg, M.},
       title={Computing {H}eegner points arising from {S}himura curve
  parametrizations},
        date={2009},
     journal={Clay Mathematics Proceedings},
      volume={8},
}

\bib{Molina2012b}{article}{
      author={Molina, S.},
       title={Ribet bimodules and specialization of {H}eegner points},
        date={2012},
     journal={Israel J. of Math.},
      volume={189},
       pages={1\ndash 38},
}

\bib{NeukirchBook}{book}{
      author={Neukirch, J.},
       title={Algebraic number theory},
      series={Grundlehren der mathematischen Wissenschaften},
   publisher={Springer},
        date={1999},
      volume={322},
}

\bib{Nualart_tesis}{book}{
      author={Nualart, J.},
       title={On the hyperbolic uniformization of {S}himura curves with an
  {A}tkin-{L}ehener quotient of genus $0$},
   publisher={PhD Thesis},
        date={2015},
}

\bib{Shimura1967}{article}{
      author={Shimura, G.},
       title={Construction of class fields and {Z}eta functions of algebraic
  curves},
        date={1967},
     journal={Annals of Mathematics},
      volume={85},
      number={1},
       pages={58\ndash 159},
         url={http://www.jstor.org/stable/1970526},
}

\bib{Vigneras1980}{book}{
      author={Vigneras, M.~F.},
       title={Arithm\'{e}tique des alg\`{e}bres de quaternions},
      series={Lecture Notes in Math.},
   publisher={Springer},
        date={1980},
      volume={800},
         url={http://books.google.es/books?id=jrUZAQAAIAAJ},
}

\bib{Voight2006}{article}{
      author={Voight, J.},
       title={Computing {CM} points on {S}himura curves arising from cocompact
  arithmetic triangle groups},
        date={2006},
     journal={Chapter in {A}lgorithmic {N}umber {T}heory, Lecture Notes in
  Computer Science},
      volume={4076},
       pages={406\ndash 420},
}

\bib{VoightWillis2011}{article}{
      author={Voight, J.},
      author={Willis, J.},
       title={Computing power series expansions of modular forms},
        date={2011},
     journal={Chapter in Computations with modular forms},
      volume={6},
       pages={331\ndash 361},
}

\end{biblist}
\end{bibdiv}

\end{document}